\documentclass[11pt]{article}

\usepackage[margin=1in]{geometry}
\usepackage{amsmath,amssymb,amsthm,mathtools}
\usepackage{tikz}
\usetikzlibrary{arrows.meta,decorations.pathreplacing}
\usepackage[
  colorlinks=true,
  linkcolor=blue,
  urlcolor=blue,
  citecolor=blue
]{hyperref}
\usepackage[capitalize,nameinlink]{cleveref}

\hypersetup{
  pdftitle={Strong freezing in the binary perceptron model},
  pdfauthor={Shuyang Gong, Shuangping Li}
}

\newtheorem{theorem}{Theorem}[section]
\newtheorem{lemma}[theorem]{Lemma}
\newtheorem{proposition}[theorem]{Proposition}

\numberwithin{equation}{section}

\newcommand{\PP}{\mathbb P}
\newcommand{\EE}{\mathbb E}

\newcommand{\ind}{\mathbf 1}
\newcommand{\ip}[2]{\langle #1,#2\rangle}

\newcommand{\norm}[1]{\lVert #1\rVert}

\newcommand{\Pois}{\operatorname{Pois}}
\newcommand{\simr}{\mathrel{\sim_r}}

\usepackage[dvipsnames]{xcolor}

\title{Strong freezing in binary perceptron models}

\author{%
  Shuyang Gong\textsuperscript{1,2}%
  \thanks{Email: \href{mailto:shuyang.gong@yale.edu}{\texttt{shuyang.gong@yale.edu}}}%
  \quad
  Shuangping Li\textsuperscript{2}%
  \thanks{Email: \href{mailto:shuangping.li@yale.edu}{\texttt{shuangping.li@yale.edu}}}\\[6pt]
  {\small \textsuperscript{1}Institute for Foundations of Data Science, Yale University}\\
  {\small \textsuperscript{2}Department of Statistics and Data Science, Yale University}
}

\date{\today}

\begin{document}
\maketitle

\begin{abstract}
The binary perceptron model concerns the set of solutions obtained by
intersecting random half-spaces. 
One of the most striking physics predictions for this model,
going back to Krauth and M\'ezard \cite{KM89} and sharpened in
\cite{HWK13,HK14}, is strong freezing: typical solutions are at Hamming
distance of order $n$ from every other solution. This phenomenon has attracted a lot of attention over the years, partly because of its intricate connections to statistical physics and computational complexity. We prove strong freezing at every positive constraint density for a general class of perceptron models and for any disorder with independent entries of mean zero and unit variance. More precisely, with probability at least
$1-e^{-c\sqrt n}$, all but an $e^{-c\sqrt n}$ fraction of the solutions have
no other solution within Hamming distance $\delta n$. In particular, our
result includes the asymmetric binary perceptron at every margin, for which
strong freezing has been a long-standing open problem \cite{KM89,HWK13,HK14}, and the symmetric binary perceptron, for which we recover the strong freezing results of
Perkins and Xu \cite{PX24} and of Abbe, Li and Sly \cite{ALS21}. 
Our proof crucially relies on replacing the number of constraints by a
Poisson variable, which makes the numbers of constraints cutting out
different subsets of the hypercube independent. Since each constraint can only
remove solutions, the FKG inequality then decouples the partition function
from the event that two nearby configurations are both solutions.
\end{abstract}

\section{Introduction}\label{sec:intro}

A basic problem in high-dimensional probability and geometry concerns the
intersection of a fixed set $A\subseteq\mathbb R^n$ with $m$ independent
random half-spaces. One asks how large $m$ can be before this intersection
becomes empty, and what the intersection looks like before then. The problem
arises from the perceptron of Rosenblatt \cite{Ros58}, a model of a single
neuron that classifies an input vector according to the sign of a weighted
sum of its coordinates. The perceptron stores a collection of random
patterns, that is, random inputs with prescribed outputs, if its weights are
chosen so that every pattern is classified correctly \cite{Hop82,EV01}. Each
pattern confines the weight vector to a half-space, so the weight vectors
that store all $m$ patterns, which we call solutions, form the intersection
of $m$ independent random half-spaces with the set $A$ of admissible weights.
When $A$ is the sphere and the half-spaces pass through the origin, this is a
problem of convex geometry, and Cover \cite{Cov65} showed by counting the
regions of a hyperplane arrangement that $n$ weights can store about $2n$
random patterns. The spherical perceptron, including its nonconvex variants,
has since been the subject of many works
\cite{ST03,Sto13a,Tal11,Sto13b,FP16,FPSUZ17,AS22,MZZ24}. In this paper we
consider the binary perceptron, in which $A$ is the discrete hypercube
$\{-1,+1\}^n$ and the solutions are the vertices of the hypercube that lie
in a random polyhedron. Convexity is then lost, and the model becomes a random
constraint satisfaction problem in which every constraint involves all $n$
variables. Gardner \cite{Gar87,Gar88} initiated the study of such solution
spaces with the replica method of statistical physics \cite{MPV87}, and for
binary weights Krauth and M\'ezard \cite{KM89} used this method to predict
the storage capacity. Their analysis also points to a frozen structure of the
solution space, which is the subject of the present paper.

Over the past decades, the binary perceptron has become a central model well
beyond the theory of neural networks. It has been studied extensively in
statistical physics and machine learning
\cite{BZ06,BBBZ07,HWK13,HK14,BILSZ15,BBCILSZ16} and in probability theory
\cite{KR98,Tal99,Tal00,Tal11,Xu21,BNSX22,NS23,ALS21,PX24,DS25,Hua24}. In
theoretical computer science, it is closely related to discrepancy
minimization \cite{Spe85,Ban10,LM15,BS20,ALS21b,GKPX23}, which in turn has
applications to causal inference \cite{TMR20,HSSZ24}. It also serves as a
testing ground for questions about algorithms and computational complexity
\cite{ALS22,GKPX22,GKPX23,LSZ25,VV25,AG25,GHLS26,JL26}. At the same time, the
binary perceptron is notoriously difficult to analyze rigorously. Even its
storage capacity at zero margin, predicted by Krauth and M\'ezard in 1989, was
established only recently, after more than three decades of work
\cite{DS25,Hua24,Shm26}.

\subsection{Perceptron models}

Fix $\alpha>0$ and a Borel set $K\subseteq\mathbb R$, and let $m=\lfloor\alpha n\rfloor$. Let $\nu$ be a probability measure on $\mathbb R$ which has zero mean and unit variance, and let
the disorder $G\in\mathbb R^{m\times n}$ have i.i.d.\ entries $g_{ai}$ with law $\nu$, so that
\begin{equation}\label{eq:moments}
 \EE g_{11}=0,\qquad \EE g_{11}^2=1 .
\end{equation}
The binary perceptron model with activation function $\mathbf{1}_K$ has solution set
\begin{equation}
    \label{eq:model}
    S(G)=\big\{\sigma\in\{+1,-1\}^n:\, \langle g_a,\sigma\rangle/\sqrt{n}\in K\text{ for }a=1,\dots,m\big\}\,,\quad Z(G)=|S(G)|\,,
\end{equation}
where $g_a$ denotes the rows of matrix $G$ and $\langle\cdot,\cdot\rangle$ denotes the Euclidean inner product. $Z(G)$ counts the number of solutions, also called the partition function. We call $\alpha$, the number of constraints per dimension, the \emph{constraint density}. 

The only assumption we make on $K$ is that the activation function
$\mathbf{1}_K$ has a jump discontinuity, a condition we make precise in
\eqref{eq:boundary} below. This covers the two standard perceptron models.
For $K=[\kappa,\infty)$ with $\kappa\in\mathbb R$, each constraint confines
$\sigma$ to the half-space $\{x\in\mathbb R^n:\langle g_a,x\rangle\geq\kappa\sqrt n\}$,
so $S(G)$ is the intersection of the hypercube with $m$ independent random
half-spaces, as at the beginning of the introduction. This is the
\emph{asymmetric binary perceptron} (ABP) with margin $\kappa$, the original
binary perceptron studied in \cite{GD88,KM89,KR98,Tal99}. For
$K=[-\kappa,\kappa]$ with $\kappa>0$, each half-space is replaced by a
symmetric slab, and the model is the \emph{symmetric binary perceptron} (SBP)
introduced by Aubin, Perkins and Zdeborov\'a \cite{APZ19}. In both models
$\mathbf{1}_K$ jumps at $\kappa$. In general, since $\mathbf{1}_K$ takes only the values
$0$ and $1$, a jump discontinuity of $\mathbf{1}_K$ at a point $\kappa$ means
that there exists $\eta>0$ such that one of the intervals
$(\kappa-\eta,\kappa)$ and $(\kappa,\kappa+\eta)$ is contained in $K$ and the
other is disjoint from $K$:
\begin{equation}\label{eq:boundary}
 \begin{gathered}
 (\kappa-\eta,\kappa)\subseteq K\quad\text{and}\quad
 (\kappa,\kappa+\eta)\cap K=\varnothing,\\
 \text{or}\\
 (\kappa-\eta,\kappa)\cap K=\varnothing\quad\text{and}\quad
 (\kappa,\kappa+\eta)\subseteq K .
 \end{gathered}
\end{equation}
Condition \eqref{eq:boundary} holds for every finite union of intervals for which
both $K$ and $\mathbb R\setminus K$ have positive Lebesgue measure. It also
holds for a union of infinitely many open intervals as soon as one of them
has an endpoint at positive distance from the others, as for
$K=\bigcup_{\ell\in\mathbb Z}(2\ell,2\ell+1)$.

A central question for perceptron models is the storage capacity, which
asks how many random patterns the perceptron can store, that is, how large
$m$ can be relative to $n$ while the solution set $S(G)$ is still nonempty.
In the proportional regime, one expects a
critical density $\alpha_\star$ such that $S(G)$ is nonempty with
high probability for $\alpha<\alpha_\star$ and empty with high
probability for $\alpha>\alpha_\star$. For the ABP with margin $\kappa$, Krauth and
M\'ezard \cite{KM89} used the replica method to predict an explicit value of
$\alpha_\star=\alpha_\star(\kappa)$, for instance $\alpha_\star(0)\approx0.833$. Unlike
the spherical perceptron, whose storage capacity was determined rigorously
for every $\kappa\geq0$ by exploiting convexity \cite{ST03,Sto13a}, the ABP
is a nonconvex problem, and proving the prediction of Krauth and M\'ezard
turned out to be far more difficult. At
$\kappa=0$ the first moment method gives only $\alpha_\star(0)\leq1$, and
even improving on this trivial bound required substantial work. Kim and
Roche \cite{KR98} and Talagrand \cite{Tal99} showed that there is
$\varepsilon>0$ such that $S(G)$ is nonempty with high probability for
$\alpha<\varepsilon$ and empty with high probability for
$\alpha>1-\varepsilon$. Using a rigorous version of the cavity method,
Talagrand \cite{Tal00,Tal11} also verified the replica-symmetric formula for
the limit of $n^{-1}\log Z(G)$ at small densities. Nevertheless, the
prediction of Krauth and M\'ezard remained out of reach for decades. A
breakthrough came with the work of Ding and Sun \cite{DS25}, who proved that
this prediction is a lower bound for the storage capacity at $\kappa=0$.
This was followed by a series of further works
\cite{Xu21,BNSX22,NS23,AT24,Nak25}, including sharp threshold results
\cite{Xu21,NS23} and the improved upper bound $0.847$ at $\kappa=0$
\cite{AT24}, and Huang
\cite{Hua24} finally proved the matching upper bound, which confirms the
prediction of Krauth and M\'ezard at $\kappa=0$ \cite{DS25,Hua24,Shm26}.

When $K=-K$, as for the SBP, the problem is more
tractable. The solution set is then invariant under
$\sigma\mapsto-\sigma$, so its barycenter is zero, and the first and second moment methods determine the storage
capacity of the SBP, which is given by the first moment computation
\cite{APZ19,PX24,ALS21}. For the ABP, in contrast, the barycenter of $S(G)$
is expected to be macroscopic and random, and $Z(G)$ is expected to be
exponentially smaller than $\EE Z(G)$ at every density, so the moment
method fails \cite{Hua24}. This is why the first moment bound
$\alpha_\star(0)\leq1$ is far from the truth, and why the proofs in
\cite{DS25,Hua24} rely on conditioning on a proxy for the barycenter,
motivated by the TAP equations from statistical physics. The symmetry also
allows a more refined analysis of the SBP than is currently available for
the ABP, including a lognormal limit law for the number of solutions
\cite{ALS21}, a precise description of the critical window
\cite{Alt23,SS23}, and results on the performance of algorithms
\cite{ALS22,GKPX22,GKPX23,BAKZ24}.

\subsection{Strong freezing}
\label{sec:strong-freezing}

Beyond the storage capacity, one of the most striking predictions of
statistical physics for the binary perceptron concerns the geometry of its
solution space. Using the replica method, Krauth and M\'ezard \cite{KM89}
predicted that, for the ABP at every density below the storage capacity,
typical solutions lie in clusters of $2^{o(n)}$ solutions and distinct
clusters are separated by linear Hamming distances, a structure known in
physics as frozen one-step replica symmetry breaking. Huang, Wong and Kabashima \cite{HWK13} and Huang
and Kabashima \cite{HK14} sharpened this prediction. They computed the
number of solutions at a given Hamming distance from a typical solution,
and concluded that a typical solution has no other solution within Hamming
distance of order $n$ (see \cref{fig:strong_freezing}). The cluster of a typical solution then consists of
the solution alone, and all of its coordinates are frozen. We refer to this
property as \emph{strong freezing}; see also \cite{ZM08a,ZM08b,ZK11} for
related phenomena in other random constraint satisfaction problems.
Strong freezing is predicted to hold at every positive density, in sharp
contrast with sparse random constraint satisfaction problems such as random
$k$-SAT and the coloring of random graphs, in which clustering and freezing
emerge only at densities bounded away from zero
\cite{MMZ05,ART06,KMRSZ07,ZK07,AC08,Mol18}. Strong freezing is also
closely tied to algorithmic tractability, which we discuss in more detail
in~\cref{sec:algorithms}. 
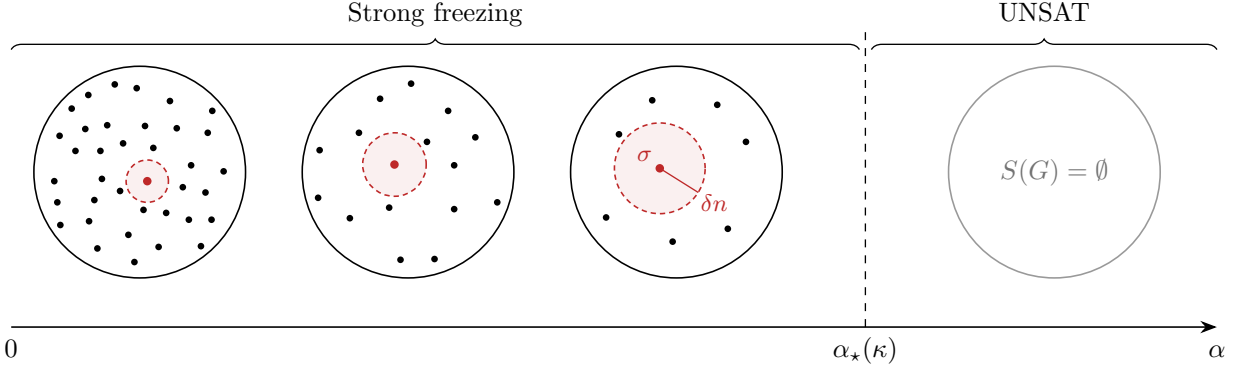
\begin{figure}
    \centering
    \definecolor{sfred}{RGB}{190,40,40}%
\begin{tikzpicture}[
  font=\small,
  disk/.style={line width=0.6pt},
  ball/.style={draw=sfred, fill=sfred!7, line width=0.6pt, dash pattern=on 2pt off 1.5pt},
  brace/.style={decorate, decoration={brace, amplitude=4pt}, line width=0.5pt},
  title/.style={midway, above=5pt, text height=1.7ex, text depth=0.4ex},
  axlabel/.style={below, text height=1.6ex, text depth=0.5ex},
]
\begin{scope}[shift={(1.70,0)}]
  \draw[disk] (0,0) circle (1.4);
  \path[ball] (0.10,-0.12) circle (0.28);
  \foreach \x/\y in {-0.26/-0.25, -0.04/1.11, -0.85/0.28, 0.25/-0.99, -1.06/0.48, 1.11/0.01, -0.15/-0.83, 0.57/-0.20, -0.90/0.84, -0.52/0.28, 0.07/0.61, 0.68/0.09, -1.09/-0.40, 0.40/0.94, 0.81/-0.98, -0.22/0.38, -0.56/-1.01, -0.60/-0.37, -0.72/0.57, -1.05/-0.70, 0.87/-0.27, -0.07/-1.19, -0.43/0.62, -0.67/-0.65, 0.35/-0.54, 0.90/0.52, -0.50/-0.09, -1.13/-0.12, -0.68/1.02, 0.18/0.32, -0.33/1.16, 0.96/0.81, 0.95/-0.63, 0.49/0.58, 0.65/-0.63, 0.05/-0.50}
    \fill (\x,\y) circle (1.25pt);
  \fill[sfred] (0.10,-0.12) circle (1.6pt);
\end{scope}
\begin{scope}[shift={(5.25,0)}]
  \draw[disk] (0,0) circle (1.4);
  \path[ball] (-0.18,0.10) circle (0.42);
  \foreach \x/\y in {0.25/0.40, 1.18/-0.40, -0.10/-1.16, -0.37/0.97, -0.77/-0.61, 0.53/0.81, -0.65/0.52, -1.19/-0.34, 0.61/-0.49, 0.61/0.09, 0.35/-1.15, -1.17/0.29, 0.89/0.46, 0.04/1.17, -0.25/-0.47}
    \fill (\x,\y) circle (1.25pt);
  \fill[sfred] (-0.18,0.10) circle (1.6pt);
\end{scope}
\begin{scope}[shift={(8.80,0)}]
  \draw[disk] (0,0) circle (1.4);
  \path[ball] (-0.22,0.05) circle (0.60);
  \draw[sfred, line width=0.5pt] (-0.22,0.05) -- (0.29,-0.27);
  \node[sfred, font=\footnotesize] at (0.50,-0.40) {$\delta n$};
  \node[sfred, font=\footnotesize] at (-0.42,0.22) {$\sigma$};
  \foreach \x/\y in {-0.76/0.50, 0.92/0.40, 0.68/-0.75, -0.05/-0.92, -0.93/-0.60, -0.32/0.95, 0.54/0.89}
    \fill (\x,\y) circle (1.25pt);
  \fill[sfred] (-0.22,0.05) circle (1.6pt);
\end{scope}
\begin{scope}[shift={(13.80,0)}]
  \draw[disk, draw=black!40] (0,0) circle (1.4);
  \node[black!55] at (0,0) {$S(G)=\emptyset$};
\end{scope}
\draw[line width=0.6pt, -{Stealth[length=2.2mm]}] (0,-2.05) -- (15.95,-2.05);
\node[axlabel] at (0,-2.05) {$0$};
\node[axlabel] at (15.95,-2.05) {$\alpha$};
\draw[dashed, line width=0.5pt] (11.3,-2.05) -- (11.3,1.87);
\node[axlabel] at (11.3,-2.05) {$\alpha_\star(\kappa)$};
\draw[brace] (0,1.62) -- (11.22,1.62) node[title] {Strong freezing};
\draw[brace] (11.38,1.62) -- (15.95,1.62) node[title] {UNSAT};
\end{tikzpicture}%
    \caption{The solution space of the ABP as predicted in~\cite{HWK13,HK14} (adapted from~\cite[Fig.~1(c)]{HK14}). Each disk represents the hypercube $\{-1,+1\}^n$ and each dot a solution. For every $0<\alpha<\alpha_\star(\kappa)$, a typical solution $\sigma$ (red) has no other solution within Hamming distance $\delta n$ (dashed circle), where $\delta=\delta(\alpha,\kappa)>0$; beyond the satisfiability threshold $\alpha_\star(\kappa)$ there are no solutions with high probability. \Cref{thm:main} confirms the strong freezing part of this picture for every $\alpha>0$ and $\kappa\in\mathbb R$.}
    \label{fig:strong_freezing}
\end{figure}

Proving strong freezing rigorously is a challenging problem. As recalled
in the previous subsection, even determining the storage capacity of the
ABP has required decades of work, and strong freezing calls for a much
finer understanding of the solution space. The standard route to
properties of a uniformly random solution passes through the planted model
\cite{AC08,KZ09}, in which $\sigma$ is drawn uniformly from
$\{-1,+1\}^n$ and $G$ is then conditioned on $\sigma\in S(G)$. Under the
planted law, the isolation of the planted solution can be established by a
first moment computation. The difficulty lies in transferring this
information to a uniformly random solution of a typical instance. Since
the law of $G$ under the planted model has density $Z(G)/\EE Z(G)$ with
respect to its original law, such a transfer requires showing that $Z(G)$
is typically not much smaller than $\EE Z(G)$, which amounts to a
fine-grained control of the fluctuations of the partition function.

For this reason, rigorous results on strong freezing have so far been
limited to symmetric models such as the SBP, whose partition function is
much better behaved thanks to the symmetry. Strong freezing for the SBP was
proved at every density below the storage capacity along two different
routes. Perkins and Xu \cite{PX24} showed, by an inductive argument in
which the constraints are added one at a time, that $\log Z(G)$ is
concentrated around $\log\EE Z(G)$ on the scale $\log n$, which suffices
to transfer the isolation of the planted solution to the original law.
Abbe, Li and Sly
\cite{ALS21} proved strong freezing through contiguity. They showed that
$Z(G)/\EE Z(G)$ converges in distribution to a lognormal limit, by
developing a dense analogue of the small subgraph conditioning method
\cite{RW94,Jan95}, and deduced that the planted and original laws are
mutually contiguous, as conjectured in \cite{APZ19}, in the sense that every event of
vanishing probability under the planted law also has vanishing probability
under the original law. Strong freezing then follows directly from the
isolation of the planted solution.

For asymmetric $K$, strong freezing has remained open, even for the ABP,
which is the model of the original predictions \cite{KM89,HWK13,HK14} and
for which strong freezing is widely believed \cite[Section~1.1]{GHLS26}.
Both approaches above break down for the ABP, as anticipated in
\cite{PX24}. Indeed, $Z(G)$ is expected
to be exponentially smaller than $\EE Z(G)$, and the planted law is not
even contiguous to the original law at any density $\alpha>0$, since
conditioning on $\sigma\in S(G)$ shifts the mean of every row $g_j$ by a
positive multiple of $\sigma/\sqrt n$, which can be detected from
$\norm{\sum_jg_j}$.

In this paper we establish strong freezing for general perceptron models,
with no symmetry assumption on $K$ and for general disorder, and in
particular for the ABP at every margin. Instead of comparing the planted and
original laws, we replace the number of constraints by a Poisson variable,
so that the numbers of constraints cutting out different subsets of the
hypercube become independent. Since each constraint can only remove
solutions, the FKG inequality then decouples the partition function from
the event that two nearby configurations are both solutions, and the
fluctuations of $Z(G)$ never enter the argument (see \cref{sec:outline}).

To state strong freezing precisely, let
$d_H(\sigma,\tau)=|\{i:\sigma_i\neq\tau_i\}|$ be the Hamming distance.
For an integer $r\geq1$, the solutions that have another solution within
distance $r$ form the set
\[
 A_r(G)=\bigl\{\sigma\in S(G):\ 1\leq d_H(\sigma,\tau)\leq r
        \text{ for some }\tau\in S(G)\bigr\},
\]
and a solution outside $A_r(G)$ is called isolated up to distance $r$. The
fraction of solutions that are not isolated is
\[
 R_r(G)=\frac{|A_r(G)|}{Z(G)}\quad\text{if }Z(G)>0,
 \qquad R_r(G)=0\quad\text{if }Z(G)=0,
\]
which, when $Z(G)>0$, is the conditional probability given $G$ that a
uniformly random solution is not isolated up to distance $r$. Strong
freezing at density $\alpha$ means that $R_{\lfloor\delta n\rfloor}(G)\to0$
in probability for some $\delta>0$.

Our main result proves strong freezing for general perceptron models at every density $\alpha>0$.

\begin{theorem}\label{thm:main}
Assume \eqref{eq:moments} and \eqref{eq:boundary}. Then for every $\alpha>0$ there exist $\delta\in(0,1/4]$, $c>0$ and $n_0\in\mathbb N$, depending only on $\nu$, $\alpha$ and $K$, such that for all $n\geq n_0$
\[
 \EE R_{\lfloor\delta n\rfloor}(G)\leq e^{-2c\sqrt n}
 \qquad\text{and}\qquad
 \PP\bigl(R_{\lfloor\delta n\rfloor}(G)>e^{-c\sqrt n}\bigr)
 \leq e^{-c\sqrt n}.
\]
\end{theorem}

\Cref{thm:main} establishes strong freezing at every density for every set
$K$ satisfying \eqref{eq:boundary} and every disorder law $\nu$ satisfying
\eqref{eq:moments}. In particular, it covers the ABP at every margin
$\kappa\in\mathbb R$ and the SBP at every margin $\kappa>0$, for Gaussian,
Rademacher and any other centered disorder with unit variance. With
probability at least $1-e^{-c\sqrt n}$, all but an $e^{-c\sqrt n}$ fraction of
the solutions have no other solution within Hamming distance
$\lfloor\delta n\rfloor$. Thus, if two solutions are declared adjacent when
their Hamming distance is at most $\delta n$, the connected component, or
cluster, of a typical solution consists of this solution alone, and all of its
coordinates are frozen, in the sense that reversing any nonempty set of at
most $\delta n$ of them violates some constraint. For the ABP, this confirms the prediction of \cite{HWK13,HK14} illustrated in
\cref{fig:strong_freezing}, which sharpens the frozen picture of Krauth and
M\'ezard \cite{KM89}.
For the SBP, it recovers the strong freezing results of
\cite{PX24,ALS21}. Since \cref{thm:main} holds for every $\alpha>0$, it covers
every density at which solutions exist, and it requires no information on the
storage capacity.

We note that the fraction of non-isolated solutions in \cref{thm:main} is
stretched exponentially small, namely at most $e^{-c\sqrt n}$, rather than
exponentially small. This rate is governed by solutions with very close
neighbors, and it is natural. Let $\kappa$ be as in \eqref{eq:boundary}. For a
solution $\sigma$, order $\sqrt n$ of the $m$ constraints are nearly tight, in
the sense that $\ip{g_j}{\sigma}$ lies within $O(1)$ of $\kappa\sqrt n$.
Reversing a single coordinate changes each $\ip{g_j}{\sigma}$ by an amount of
order one, and therefore violates each nearly tight constraint with
probability bounded away from zero. A neighbor of $\sigma$ at distance one is
thus a solution only with probability $e^{-\Theta(\sqrt n)}$. In our argument,
this is the contribution of $k=1$ in \eqref{eq:sumk}, which suggests that the
rate $e^{-c\sqrt n}$ in \cref{thm:main} is of the correct order.

At linear distances, in contrast, our argument yields an exponential bound.
For every fixed $\theta\in(0,\delta)$, restricting the sum in \eqref{eq:sumk}
to $k\ge\theta n$ and arguing as in the proof of \cref{thm:main} shows that,
with probability $1-e^{-\Omega(n)}$, all but an $e^{-\Omega(n)}$ fraction of
the solutions have no other solution at Hamming distance between $\theta n$
and $\delta n$, where the implicit constants also depend on $\theta$.
Consequently, the solutions that have exponentially many other solutions
within distance $\delta n$ form an exponentially small fraction of $S(G)$, in
line with the prediction that the dense regions of solutions found in
\cite{BILSZ15} are exponentially rare.

\subsection{Outline of the proof}\label{sec:outline}

Every solution in $A_r(G)$ has another solution within distance $r$, so
\[
 \EE R_r(G)\le\sum_{\sigma\in\{-1,+1\}^n}\
 \sum_{\tau:\,1\le d_H(\sigma,\tau)\le r}
 \EE\Bigl[\frac{\ind\{\sigma\in S(G),\,\tau\in S(G)\}}{Z(G)}\Bigr].
\]
The difficulty is the factor $1/Z(G)$, which is correlated with the events
$\{\sigma\in S(G)\}$ and $\{\tau\in S(G)\}$. Controlling this correlation
directly would require information on the fluctuations of $Z(G)$, which is
not available for asymmetric $K$ such as the ABP.

We view each constraint as a random subset of $\{-1,+1\}^n$. For
$g\in\mathbb R^n$, the constraint with vector $g$ is satisfied exactly by the
configurations in
\begin{equation}\label{eq:half-space-by-g}
 C(g)=\{\sigma\in\{-1,+1\}^n:\,\langle g,\sigma\rangle/\sqrt{n}\in K\},
\end{equation}
so that $S(G)=\bigcap_{a=1}^m C(g_a)$ by \eqref{eq:model}. Our key step is to replace the number $m$ of constraints by a Poisson variable
$N(\lambda)$ with mean $\lambda$. By the Poisson colouring theorem \cite{Kin93}, the
numbers of constraints $C(g_j)$ equal to each given subset $T$ of
$\{-1,+1\}^n$ are then independent Poisson variables. This extra independence
decouples the events we need. Write $S$ and $Z$ for the solution set and its size with $N(\lambda)$ constraints. The event $\{\sigma\in S\}$ only requires that no
constraint excludes $\sigma$, so it depends only on the counts of subsets $T\not\ni\sigma$,
and conditioning on it leaves the counts of subsets $T\ni\sigma$ independent. On this
event, both $\ind\{\tau\in S\}$ and $1/Z$ are functions of the independent
indicators that some constraint equals $T$, for $T\ni\sigma$. Each such
constraint can only remove solutions, so $\ind\{\tau\in S\}$ is nonincreasing
and $1/Z$ is nondecreasing in these indicators. The FKG inequality
\cite{Har60,FKG71} then shows that they are negatively correlated, which gives
\[
 \EE\Bigl[\frac{\ind\{\sigma\in S,\,\tau\in S\}}{Z}\Bigr]
 \le e^{-\lambda\PP(\sigma\in C(g_1),\,\tau\notin C(g_1))}\,
 \EE\Bigl[\frac{\ind\{\sigma\in S\}}{Z}\Bigr].
\]
Exchanging $\sigma$ and $\tau$ gives the same bound with
$\PP(\tau\in C(g_1),\,\sigma\notin C(g_1))$ in the exponent, and the two
exponents add up to $\lambda\Delta(\sigma,\tau)$, where $\Delta(\sigma,\tau)$
is the probability that a single constraint contains exactly one of $\sigma$
and $\tau$. Taking the geometric mean of the two bounds and then the
arithmetic mean of the weights, we bound each term by
$\frac12e^{-\lambda\Delta(\sigma,\tau)/2}\bigl(\EE[\ind\{\sigma\in S\}/Z]+\EE[\ind\{\tau\in S\}/Z]\bigr)$.
Since these weights sum to at most one over $\sigma$, $Z$ drops out after
summing, taking $\lambda=m$ and returning from $N$ to $m$ constraints, which costs a factor
$3\sqrt m$, leads to \eqref{eq:poimeanm}:
\[
 \EE R_r(G)\le3\sqrt m\max_{\sigma}\sum_{\tau:\,1\le d_H(\sigma,\tau)\le r}
 e^{-m\Delta(\sigma,\tau)/2}.
\]

It remains to bound $\Delta(\sigma,\tau)$ from below. Let
$k=d_H(\sigma,\tau)$. The inner products $\ip g\sigma$ and $\ip g\tau$ share the $n-k$ coordinates on which $\sigma$ and $\tau$ agree, and they differ by a sum of $k$ terms, which is of order $\sqrt k$. For the two inner products to fall on different sides of the boundary of $K$, their common part, which is spread over a scale of $\sqrt n$, must land in a window of width of order $\sqrt k$. This has probability of order $\sqrt{k/n}$, which is the lower bound $\Delta(\sigma,\tau)\ge\beta\sqrt{k/n}$ of \cref{lem:comparison} when $k/n$ is small. We make this heuristic rigorous in \cref{sec:proof}. Plugging this bound into \eqref{eq:poimeanm}, each $\tau$ at distance $k$ contributes at most $e^{-c\sqrt{nk}}$, which beats the $\binom nk$ choices of $\tau$ when $k/n$ is small, and summing over $k\le\delta n$ proves \cref{thm:main}.

\subsection{Discussion and related work}
\paragraph{Freezing in random constraint satisfaction problems.}
The geometry of the solution space of random constraint satisfaction
problems has been studied extensively in statistical physics by means of
the cavity method \cite{MPZ02,MZ02,MRZ03,MM09}. As the density of constraints
increases, the solution space of sparse problems such as random $k$-SAT
and the coloring of random graphs is predicted to undergo several phase
transitions before the satisfiability threshold
\cite{MMZ05,KMRSZ07,ZK07,MRS08}: it shatters into exponentially many
clusters separated by linear Hamming distances, the uniform measure on
solutions condenses onto a small number of these clusters, and a linear
fraction of the variables becomes frozen in typical clusters, in the
sense that these variables take the same value at every solution of the
cluster \cite{Sem08}. Much of this picture has been confirmed rigorously,
including clustering and the existence of frozen variables in random
$k$-SAT, hypergraph $2$-coloring and graph coloring when $k$ is large
\cite{ART06,AC08} and in random XORSAT \cite{CDMM03,IKKM15,AcMo15},
the exact location of the freezing threshold for the coloring of random
graphs and for random NAE-SAT and hypergraph $2$-coloring
\cite{Mol18,MR13}, and the condensation and satisfiability thresholds
predicted by the cavity method \cite{BCHRV16,DSS16nae,DSS22}. In all of
these problems, however, freezing is partial and appears only at densities
bounded away from zero. At small densities the solution space is well
connected \cite{AC08}, and at every density a linear number of variables
are not involved in any constraint, so that every cluster contains
exponentially many solutions \cite[Remark~3]{AC08}.

Strong freezing, in which the cluster of a typical solution consists of
that solution alone, is much rarer than the partial freezing described
above. The closest notion in physics is frozen one-step replica symmetry
breaking, in which clusters have vanishing entropy density; its prototype
is the low-temperature phase of the random energy model
\cite{Der80,Der81}, and frozen glass phases of this kind were also found
in the random multi-index matching problem \cite{MMR04,MMR05}. Among random constraint
satisfaction problems, the main examples are the locked problems of
Zdeborov\'a and M\'ezard \cite{ZM08a,ZM08b}, in which every variable lies
in at least two constraints and changing a single variable of a satisfied
constraint violates it. In their clustered phase, clusters are predicted to
be isolated points, but this phase starts only at a positive density,
below which solutions are expected to be separated by Hamming distances of
logarithmic order \cite{ZM08b}. Isolation has also been proved in problems
with algebraic or arithmetic structure. In random regular XORSAT, where
every variable lies in at least three equations, the solutions form a coset
of a low-density parity-check code whose minimum distance is linear
\cite{Gal62}, and in the number partitioning problem, the near-optimal
solutions at exponentially small energies are isolated \cite{GK23} and
behave like the configurations of a random energy model
\cite{Mer00,BCMN09a,BCMN09b}. In binary perceptrons, in contrast, typical
solutions are predicted to be isolated at every positive density
\cite{KM89,HWK13,HK14,APZ19}, as proved for the SBP \cite{PX24,ALS21} and,
by \cref{thm:main}, for general perceptron models including the ABP. This reflects the dense structure of the
constraints: every constraint involves all the variables, so flipping even
a single coordinate perturbs every constraint, including the order
$\sqrt n$ nearly tight ones (see the discussion after \cref{thm:main}),
whereas in sparse problems each variable appears in a bounded number of
constraints.

\paragraph{Strong freezing and algorithms.}\label{sec:algorithms}
Strong freezing does not prevent efficient algorithms from finding
solutions of the perceptron, in contrast with the heuristic picture for
many of the problems discussed above. In those problems, heuristic
arguments suggest that freezing leads to algorithmic hardness: in locked
problems the best known algorithms fail throughout the clustered phase
\cite{ZM08a,ZM08b}, for the coloring of random graphs and random NAE-SAT
the freezing threshold has been suggested to be a precise algorithmic
barrier \cite{ZK07,Mol18,MR13}, and refined versions of this heuristic,
which take into account that efficient algorithms can find atypical
unfrozen solutions, point to the freezing of all solutions as the ultimate
barrier \cite{MPR16,BDSZ16,BRS19}. A notable exception is random XORSAT,
which is solved by Gaussian elimination even when its solutions are
frozen, although several classes of sequential local algorithms provably
fail beyond its clustering threshold, where freezing sets in
\cite{Yun24}. For the binary perceptron, Huang and Kabashima \cite{HK14}
likewise attributed the computational hardness to the isolation of
typical solutions. Nevertheless, efficient algorithms find solutions of
the ABP at positive densities, as observed empirically
\cite{BZ06,BBBZ07,Bal09,BB15} and proved rigorously at small densities
\cite{KR98,ALS22,LSZ25,JL26}, although by \cref{thm:main} typical
solutions are isolated at every density. This makes the perceptron
unusual among random constraint satisfaction problems \cite{Gam21}.

The explanation proposed in statistical physics is that efficient
algorithms do not find typical solutions. Instead, the isolated typical
solutions coexist with rare dense regions of solutions, and efficient
algorithms are attracted to these regions \cite{BILSZ15,BBCILSZ16}. Abbe,
Li and Sly \cite{ALS22} gave rigorous support to this picture: at small
densities there is a cluster of solutions, connected under single
coordinate flips, of almost maximal diameter, in which a multiscale
majority algorithm finds solutions. For the SBP this cluster contains an
exponentially small fraction of the solutions, and the same was expected
for the ABP. The results of \cite{ALS22} for the ABP are stated for Rademacher disorder,
which satisfies \eqref{eq:moments}, so \cref{thm:main} applies directly and
shows that any such structure is necessarily rare:
every solution in a cluster with at least two
elements lies in $A_{\lfloor\delta n\rfloor}(G)$, which with high
probability contains at most an $e^{-c\sqrt n}$ fraction of the
solutions.

For ABP with Gaussian disorder, combined with a result of Gong, Huang, Li and Sellke \cite{GHLS26},
\cref{thm:main} shows that every stable algorithm that succeeds with high
probability finds atypical solutions, for every $\alpha>0$ and
$\kappa\in\mathbb R$. Here an algorithm is stable if its output moves by
$o(\sqrt n)$ in Euclidean norm when $G$ is slightly perturbed; examples
include polynomials of degree $o(n/\log n)$ in the entries of $G$ and the
multistage majority algorithm of \cite{KR98} (see \cite{GKPX22}). By
\cite[Theorem~1.4]{GHLS26}, for every fixed $\iota\in(0,1)$, such an
algorithm outputs a solution isolated up to distance $\lfloor\iota n\rfloor$
only with probability $o(1)$; taking $\iota=\delta$, its output lies in the
rare set $A_{\lfloor\delta n\rfloor}(G)$ with probability $1-o(1)$.

These results suggest that the algorithmic threshold of the ABP is
governed by the rare well-connected structures rather than by the typical
solutions, and whether it coincides with the disappearance of these
structures remains open \cite[Question~1]{ALS22}. See
\cite{BILSZ16,BAKZ24,Bar25,Bar26,Sto25} for further studies of these rare
structures and \cite{GKPX22,GKPX23,JL26,HSS26} for rigorous limits of
restricted classes of algorithms.

\section{Poissonization and FKG inequality}\label{sec:sets}
In this section, we replace the number $m$ of constraints by a Poisson
random variable. This makes the numbers of constraints of different types
independent, and this independence allows us to decouple the events we need. 

Write $\Sigma_n=\{-1,+1\}^n$. Let $g_1,g_2,\ldots$ be independent random
vectors in $\mathbb R^n$ with independent entries of law $\nu$, so that
$g_1,\ldots,g_m$ have the joint law of the rows of $G$, and let
$C_j=C(g_j)$ for $j\ge1$, with $C(g)$ from \eqref{eq:half-space-by-g}. The sets $C_1,C_2,\ldots$ are then independent
and identically distributed random subsets of $\Sigma_n$. For $\lambda\in (0,\infty)$, we define $N(\lambda)\sim \Pois(\lambda)$ to be the Poisson variable with mean $\lambda$. For two binary vectors $\sigma,\tau\in \Sigma_n$, we write $\sigma\simr \tau$ if $ 1\leq d_{H}(\sigma,\tau)\leq r$. For any integer $M\geq 0$, we let 
\begin{equation*}
    S_M=\bigcap_{1\leq a\leq M} C_a\,,\quad Z(S_M)=|S_M|\,, \quad R_{r}(S_M)=\frac{|\{\sigma\in S_M:\sigma\simr \tau\text{ for some }\tau\in S_M\}|}{Z_M}\notag\,,
\end{equation*}
where the ratio is $0$ if $Z_M=0$.
Therefore, for the model \eqref{eq:model},
\[
 S(G)=S_m,\qquad Z(G)=|S_m|,\qquad R_r(G)=R_r(S_m).
\]
\Cref{prop:upper-bound-FKG} bounds $\EE R_r(G)$ in terms of the probability that $C_1$
contains exactly one of two given configurations. Its proof relies on the
FKG inequality \cite{FKG71} for product measures, which is due to Harris
\cite{Har60}. We recall it in the form in which we use it.

\begin{lemma}[FKG inequality \cite{Har60,FKG71}]
\label{lem:harrisineq}
Let $\xi_1,\ldots,\xi_p$ be independent real random variables with values in
an interval $I$, and let $f,h:I^p\to\mathbb R$ be bounded and measurable,
with $f$ coordinatewise nondecreasing and $h$ coordinatewise
nonincreasing. Then, with $\xi=(\xi_1,\ldots,\xi_p)$,
\[
 \EE[f(\xi)h(\xi)]\leq\EE f(\xi)\,\EE h(\xi).
\]
\end{lemma}
Then we present our key proposition.
\begin{proposition}\label{prop:upper-bound-FKG}
    For every positive integer $m$, $r$ and every $\lambda>0$, we have
    \begin{equation}\label{eq:each-lambda}
        \mathbb ER_r(G) \leq \frac{1}{\mathbb P\big(\Pois(\lambda)=m\big)} \max_{\sigma\in \Sigma_n} \sum_{\tau:\,\sigma\simr \tau} e^{-\lambda\Delta(\sigma,\tau)/2}\,,
    \end{equation}
    where $\Delta(\sigma,\tau)=\mathbb P\big(\mathbf 1\{\sigma\in C_1\}\neq \mathbf 1\{\tau\in C_1\}\big)$. In particular, 
    \begin{equation}\label{eq:poimeanm}
        \mathbb E R_r(G)\leq 3\sqrt{m}\max_{\sigma\in \Sigma_n} \sum_{\tau:\,\sigma\simr \tau} e^{-m\Delta(\sigma,\tau)/2}\,.
    \end{equation}
\end{proposition}
\begin{proof}
    For $N(\lambda)\sim \Pois(\lambda)$, we first show 
    \begin{equation}\label{eq:goal-1-FKG}
        \mathbb E R_r(S_{N(\lambda)}) \leq \max_{\sigma\in \Sigma_n} \sum_{\tau:\sigma\simr\tau} e^{-\lambda\Delta(\sigma,\tau)/2}\,,
    \end{equation}
    where the expectation is taken over both randomness of $C_1,C_2,\dots$ and $N(\lambda)$. To this end, we first explain the sources of independence introduced by Poissonization. For a subset $T\subseteq \Sigma_n$, we define 
    \begin{equation}\label{eq:count-T}
        \mathsf{Count}(T)=|\{1\leq a\leq N(\lambda):\,C_a=T\}|\,.
    \end{equation}
    Note that $C_a$'s are i.i.d random subsets, by Poisson colouring theorem, 
    we get 
    \begin{equation}
        \left\{\mathsf{Count}(T):\,T\subseteq \Sigma_n\right\}
    \end{equation}
    is a collection of independent Poisson variables with mean $\mathbb E\mathsf{Count}(T)=\lambda \mathbb P(C_1=T)$, where a Poisson variable with mean zero vanishes. Since $S_{N(\lambda)}$ is the intersection of sets $T$ with $\mathsf{Count}(T)\geq 1$, we have that for every $\sigma\in \Sigma_n$,
    \begin{equation}\label{eq:count-dependence}
        \{\sigma\in S_{N(\lambda)}\} = \{\mathsf{Count}(T)=0\text{ for any }T\text{ with }\sigma\notin T\},
    \end{equation}
    and thus $$\mathbb P\big(\sigma\in S_{N(\lambda)}\big)=\prod_{T:\sigma\notin T}\exp(-\lambda \mathbb P(C_1=T))=\exp\big(-\lambda \mathbb P(\sigma\notin C_1) \big).$$
    For $\sigma,\tau\in\Sigma_n$, let 
    \begin{equation}\label{eq:def-w-q}
        w_\sigma = \mathbb E\left[ \frac{\mathbf{1}\{\sigma\in S_{N(\lambda)}\}}{Z(S_{N_{\lambda}})} \right]\,,\quad q_{\sigma,\tau} = \mathbb E\left[ \frac{\mathbf{1}\{\sigma\in S_{N(\lambda)},\tau\in S_{N(\lambda)}\}}{Z(S_{N_{\lambda}})} \right]\,,
    \end{equation}
    where the ratio is $0$ when $Z(S_{N(\lambda)})=0$. For every $\sigma\in\Sigma_n$, the indicator that some $\tau\in S_{N(\lambda)}$ satisfies $\tau\simr\sigma$ is at most the number of such $\tau$. Therefore,
    \begin{equation}\label{eq:bound-R}
    \mathbb E R(S_{N(\lambda)})\le \mathbb E\Bigl[\frac1{Z_N}\sum_{\sigma\in S_{N(\lambda)}}\sum_{\tau\simr\sigma} \ind\{\tau\in S_{N(\lambda)}\}\Bigr]=\sum_{\sigma\in \Sigma_n}\sum_{\tau\simr\sigma}q_{\sigma,\tau}.
    \end{equation}
    Next, using FKG inequality, we will show
    \begin{equation}
        \label{eq:twosol-onesol}
        q_{\sigma,\tau}\leq \exp\big(-\lambda \mathbb P(\sigma\in C_1,\tau\notin C_1)\big)w_\sigma \text{ for all }\sigma,\tau\in \Sigma_n.
    \end{equation}
    Recall \eqref{eq:count-dependence}. The event $\sigma\in S_{N(\lambda)}$ depends only on the variables $\mathsf{Count}(T)$ with $\sigma\notin T$. Therefore, conditioning on $\{\sigma\in S_{N(\lambda)}\}$ leaves the variables $\{\mathsf{Count}(T):T\ni \sigma\}$ unchanged. On $\{\sigma\in S_{N(\lambda)}\}$, $S_{N(\lambda)}$ is exactly the intersection of sets $\{T\ni \sigma:\,\mathsf{Count}(T)\geq 1\}$. Thus, under this conditioning, we have that $\mathbf{1}\{\tau\in S_{N(\lambda)}\}$ and $1/Z(S_{N}(\lambda))$ are functions of the independent indicators $\big\{\mathbf{1}\{\mathsf{Count}(T)\geq 1\}:\,T\ni \sigma\big\}$. Noting that changing one of these indicators from $0$ to $1$ only removes elements from $S_{N(\lambda)}$, thus $\mathbf{1}\{\tau\in S_{N(\lambda)}\}$ is nonincreasing in each indicator and $1/Z(S_{N}(\lambda))$ nondecreasing in each indicator. By FKG inequality \Cref{lem:harrisineq}, we get that under this conditional law on $\{\sigma\in S_{N(\lambda)}\}$,
    \begin{equation}\label{eq:q-to-w}
        \mathbb E\left[ \frac{\mathbf{1}\{\tau\in S_{N(\lambda)}\}}{Z(S_{N(\lambda)})}\,\big|\,\sigma\in S_{N(\lambda)} \right]\leq \mathbb P(\tau\in S_{N(\lambda)}\,|\,\sigma\in S_{N(\lambda)})\cdot\mathbb E\left[\frac{1}{Z(S_{N(\lambda)})}\,\big|\,\sigma\in S_{N(\lambda)}\right]\,.
    \end{equation}
    Note that 
    \begin{equation}
    \begin{split}\label{eq:conditional-sol}
        &\mathbb P(\tau\in S_{N(\lambda)}\,|\,\sigma\in S_{N(\lambda)}) = \prod_{T:\tau\notin T,\sigma\in T} \mathbb P(\mathsf{Count}(T)=0)\\
        =&\exp\big(-\lambda \sum_{T:\tau\notin T,\sigma\in T}\mathbb P(C_1=T)\big)=\exp(-\lambda\mathbb P\big(\tau\notin C_1,\sigma\in C_1)\big)\,.
    \end{split}
    \end{equation}
    Plugging \eqref{eq:conditional-sol} into \eqref{eq:q-to-w} and then multiplying both sides of \eqref{eq:q-to-w} by $\mathbb P(\sigma\in S_{N(\lambda)})$ yields \eqref{eq:twosol-onesol}. 
    Recall the definition of $\Delta(\sigma,\tau)$ in \Cref{prop:upper-bound-FKG}. Note that $q_{\sigma,\tau}=q_{\tau,\sigma}$, by symmetrization we have
    \begin{equation}\label{eq:twosol-onesol-symmetrized}
        q_{\sigma,\tau} \leq e^{-\lambda \Delta(\sigma,\tau)/2}\sqrt{w_\sigma w_\tau}\leq \frac{1}{2} e^{-\lambda \Delta(\sigma,\tau)/2} (w_\sigma+w_\tau)\text{ for all }\sigma,\tau\in \Sigma_n\,.
    \end{equation}
    Since $\sum_{\sigma}w_\sigma = \mathbb P(Z(S_{N(\lambda)})>0)\leq 1$. Plugging \eqref{eq:twosol-onesol-symmetrized} into \eqref{eq:bound-R} yields
    \begin{equation}
    \begin{split}
        &\mathbb ER_r(S_{N(\lambda)})\leq \sum_{\sigma\in \Sigma_n}\sum_{\tau:\tau\simr \sigma}\frac{1}{2} e^{-\lambda \Delta(\sigma,\tau)/2} (w_\sigma+w_\tau) \\
        =& \sum_{\sigma\in \Sigma_n}\sum_{\tau:\tau\simr \sigma} e^{-\lambda \Delta(\sigma,\tau)/2}w_\sigma \leq \max_{\sigma\in\Sigma_n} \sum_{\tau:\tau\simr\sigma}e^{-\lambda \Delta(\sigma,\tau)/2}\,,
    \end{split}
    \end{equation}
    where the equality follows from the fact that relation $\simr$ and function $\Delta(\cdot,\cdot)$ are symmetric. This shows \eqref{eq:goal-1-FKG}. Conditional on $\{N(\lambda)=m\}$, we have
    $$\mathbb ER_r(G) = \mathbb E[R_r(S_{N(\lambda)})\,|\,N(\lambda)=m]\,.$$ Thus, we have 
    \begin{equation}
        \mathbb ER_r(G) \leq \frac{\mathbb ER_r(S_{N(\lambda)})}{\mathbb P(N(\lambda)=m)}=\frac{1}{\mathbb P\big(N(\lambda)=m\big)} \max_{\sigma\in \Sigma_n} \sum_{\tau:\,\sigma\simr \tau} e^{-\lambda\Delta(\sigma,\tau)/2}\,,
    \end{equation}
   thereby confirming \eqref{eq:each-lambda}. For the second claim, we take $\lambda=m$,
   \begin{equation}
       \mathbb P(N(m)=m) = \frac{m^m}{m!}e^{-m}\geq \frac{1}{3\sqrt{m}}.\notag
   \end{equation}
   The last inequality follows from Stirling's formula. The proof is complete.
\end{proof}

\section{Proof of \texorpdfstring{\cref{thm:main}}{Theorem 1.1}}
\label{sec:proof}
\Cref{prop:upper-bound-FKG} reduces \cref{thm:main} to a lower bound on $\Delta(\sigma,\tau)$. In this section, we therefore focus on one random subset. Throughout this section, we let $g=(X_1,\dots,X_n)$ have i.i.d coordinates with distribution $\nu$. For a Borel set $F\subseteq \mathbb R$, and $\sigma,\tau\in\Sigma_n$, the probability that $F$ contains exactly one of the two normalized inner product of $g$ with $\tau$ and $\sigma$ is
\begin{equation}\label{eq:DeltaF}
 \Delta_F(\sigma,\tau)=\PP\Bigl(
 \ind_F\bigl(\ip g\sigma/\sqrt n\bigr)
 \ne\ind_F\bigl(\ip g\tau/\sqrt n\bigr)\Bigr).
\end{equation}
With $F=K$, \eqref{eq:DeltaF} is the probability $\Delta(\sigma,\tau)$ of
Proposition 2.2, so $\Delta(\sigma,\tau)=\Delta_K(\sigma,\tau)$.
For $\sigma,\tau\in \Sigma_n$, we define $\mathsf D(\sigma,\tau)=\{i\in[n]:\,\sigma_i\neq\tau_i\}$. And let
\begin{equation}
    U=\sum_{i\notin \mathsf D(\sigma,\tau)}\sigma_iX_i\,,\quad V=\sum_{i\in \mathsf D(\sigma,\tau)}\sigma_iX_i\,.
\end{equation}
Therefore we have $\langle g,\sigma\rangle =U+V$, $\langle g,\tau\rangle = U-V$ and $U,V$ are independent. We first establish lower bounds for $F=[\kappa,\infty)$ in \Cref{sec:halfspace} and then transfer the bound to general $K$ in \Cref{sec:general}.

\subsection{Lower bound for half-lines}
\label{sec:halfspace}
For $F=[\kappa,\infty)$, we have
\begin{equation}\label{eq:deltahalfspace}
    \Delta_{[\kappa,\infty)}(\sigma,\tau) = \mathbb P\big(\kappa\sqrt{n}-|V| \leq U< \kappa\sqrt{n}+|V|\big)\,.
\end{equation}
The following lemma shows that $\Delta_{[\kappa,\infty)}(\sigma,\tau)$ is at least constant multiple of $\sqrt{d_H(\sigma,\tau)/n}$, the order of the probability that a Gaussian variable with variance $n$ falls into a fixed interval of length $\sqrt{d_H(\sigma,\tau)}$ near $\kappa\sqrt{n}$. 

\begin{lemma}\label{lem:halfline}
Assume \eqref{eq:moments} and let $\kappa\in\mathbb R$. There are
$\beta_0>0$ and $n_1\in\mathbb N$, depending only on $\kappa$ and the law of $X$, such that every $n\ge n_1$ and all $\sigma,\tau\in\{-1,+1\}^n$ with $1\le d_H(\sigma,\tau)\le n/4$ satisfy
\[
 \Delta_{[\kappa,\infty)}(\sigma,\tau)\ge\beta_0\sqrt{d_H(\sigma,\tau)/n}.
\]
\end{lemma}

\begin{proof}
    The proof treats lattice and nonlattice laws of $X$ separately, since the local limit theorems available in the two cases take different forms \cite[Section~3.5]{Dur19}. The law of a real random variable $X$ is \emph{lattice} if there exist $\gamma\in\mathbb R$ and $h>0$ such that $\PP(X\in \gamma+h\mathbb Z)=1$. The largest such $h$ is the \emph{span} of the law. Otherwise the law is \emph{nonlattice}. Now we fix $\sigma,\tau$ and write $k=d_H(\sigma,\tau)\in[1,n/4]$. We first show that $|V|$ is of order $\sqrt{k}$ with probability bounded away from zero. In the following, we let $\iota = h$ if $X$ is lattice with maximal span $h$ and $\iota=1$ if $X$ is nonlattice. 
    
    Firstly, we will show there are numbers $a\in (0,1],b\in[2,\infty)$ and $\rho>0$, depending only on the law $\nu$, such that for every $k\geq 1$ and every $\varepsilon\in \{+1,-1\}^k$ satisfy
    \begin{equation}\label{eq:uniform-prob-frombelow}
        \mathbb P\big( a\sqrt{k} \leq |\sum_{i=1}^k\varepsilon_i X_i|\leq b\sqrt{k}\text{ and }|\sum_{i=1}^k\varepsilon_i X_i|\geq \frac{\iota}{2} \big)\geq \rho.
    \end{equation}
    For large $k$, \eqref{eq:uniform-prob-frombelow} follows from the central limit theorem. Consider integers $k\to\infty$ and sign vectors $\varepsilon\in\{-1,+1\}^{k}$. The variables $\varepsilon_iX_i/\sqrt{k}$, $i\le k$, are independent and
centered, their variances add up to $1$, and every $\xi>0$ satisfies
\[
 \sum_{i=1}^{k}\EE\Bigl[\frac{X_i^2}{k}\,
 \ind\{|X_i|\ge\xi\sqrt{k}\}\Bigr]
 =\EE\bigl[X^2\,\ind\{|X|\ge\xi\sqrt{k}\}\bigr]\to 0\,,
\]
where the last convergence follows by dominated convergence theorem. Lindeberg's theorem shows that $k^{-1/2}\sum_i\varepsilon_iX_i$ converges in distribution to $\mathcal{N}(0,1)$. Therefore the
probability that $|k^{-1/2}\sum_i\varepsilon_iX_i|$ lies in $[1,2]$ converges to $\PP(1\le|\mathcal{N}(0,1)|\le2)>0$. Consequently, there is $k_0\ge\iota^2/4$ such
that
\begin{equation}\label{eq:largek}
 \PP\Bigl(\sqrt k\le\Bigl|\sum_{i=1}^k\varepsilon_iX_i\Bigr|\le2\sqrt
 k\Bigr)\ge\frac12\,\PP(1\le|\mathcal N(0,1)|\le2)
 \quad\text{for all }k\ge k_0\text{ and }\varepsilon\in\{-1,+1\}^k.
\end{equation}
\eqref{eq:largek} is also uniform in $\varepsilon\in\{+1,-1\}^k$
since otherwise there would be $k\to\infty$ and $\varepsilon\in\{-1,+1\}^{k}$ for which the probability in
\eqref{eq:largek} stays below $\PP(1\le|\mathcal{N}(0,1)|\le2)/2$.

It remains to treat $k<k_0$, where only finitely many laws of
$\sum_{i=1}^k\varepsilon_iX_i$ with $\varepsilon\in\{-1,+1\}^k$ occur. Each of
them assigns positive probability to $\{|\sum_i\varepsilon_iX_i|\geq\iota/2\}$.
If $\nu$ is nonlattice, then $\iota/2=1/2$, and $|\sum_i\varepsilon_iX_i|<1/2$
almost surely would give $\operatorname{Var}(\sum_i\varepsilon_iX_i)\le1/4<k$.
If $\nu$ is lattice with span $h$, then $\iota/2=h/2$, and
$\sum_i\varepsilon_iX_i$ takes values in an affine lattice of span $h$. The
interval $(-h/2,h/2)$ contains at most one point of this lattice, so
$|\sum_i\varepsilon_iX_i|<h/2$ almost surely would make
$\sum_i\varepsilon_iX_i$ constant, which contradicts
$\operatorname{Var}(\sum_i\varepsilon_iX_i)=k$. Hence there exists $b\ge2$,
depending only on $\nu$, such that each of these laws assigns positive
probability to $\{\iota/2\le|\sum_i\varepsilon_iX_i|\le b\}$. Therefore, we define $\rho$ to be the minimum of $\PP(1\le|\mathcal{N}(0,1)|\le2)/2$ and these positive probabilities. Let $a=\min\{1,\iota/(2\sqrt{k_0})\}$. For $k<k_0$, we have $a\sqrt{k}\leq \iota/2$ and $b\leq b\sqrt{k}$. Combined this choice of $a,b$ and \eqref{eq:largek}, \eqref{eq:uniform-prob-frombelow} follows.

Next we split $U$. The coordinates outside $\mathsf D(\sigma,\tau)$ form the two
sets $\{i\notin\mathsf D(\sigma,\tau):\sigma_i=1\}$ and
$\{i\notin\mathsf D(\sigma,\tau):\sigma_i=-1\}$. Let $J$ be the larger of them,
or the first one if they have the same size, and let $L=|J|$, therefore
$L\ge(n-k)/2\ge3n/8$. Then $U=U_J+W$ with
\[
 U_J=\sum_{i\in J}\sigma_iX_i,\qquad
 W=\sum_{i\notin\mathsf D(\sigma,\tau)\cup J}\sigma_iX_i,
\]
and $U_J$, $W$, $V$ are independent. The variable $U_J$ is a sum of $L$
independent copies of $X$ or of $-X$. Since $\EE W=0$ and
$\operatorname{Var}W=n-k-L\le n/2$, Chebyshev's inequality gives
\begin{equation}\label{eq:W-small}
 \PP\bigl(|W|\le\sqrt n\bigr)\ge\tfrac12 .
\end{equation}

Let $\varphi$ be the density of $\mathcal N(0,1)$ and
$B=\sqrt{8/3}\,(|\kappa|+b/2+1)$. We claim that there is $L_0$, depending only
on $\kappa$ and $\nu$, such that for $L\ge L_0$ every interval
$[x,x+\iota)\subseteq[-B\sqrt L,B\sqrt L]$ satisfies
\begin{equation}\label{eq:lclt}
 \PP\bigl(U_J\in[x,x+\iota)\bigr)\ge\frac{\iota\,\varphi(B)}{2\sqrt L}.
\end{equation}
If $\nu$ is nonlattice, the local limit
theorem applied to $X$ or to $-X$ gives
\[
 \sup_{x\in\mathbb R}\Bigl|\sqrt L\,\PP\bigl(U_J\in(x,x+1)\bigr)-\varphi(x/\sqrt L)\Bigr|
 \longrightarrow0\qquad(L\to\infty).
\]
If $\nu$ is lattice with span $h$, then $U_J$ takes values in an affine lattice
$\mathcal L_L$ of span $h$, and the local limit theorem gives
\[
 \sup_{z\in\mathcal L_L}\Bigl|\frac{\sqrt L}{h}\,\PP(U_J=z)-\varphi(z/\sqrt L)\Bigr|
 \longrightarrow0\qquad(L\to\infty).
\]
Each interval $[x,x+h)$ contains exactly one point of $\mathcal L_L$. All
points involved lie in $[-B\sqrt L,B\sqrt L]$, where
$\varphi(\cdot/\sqrt L)\ge\varphi(B)$. Choosing $L_0$ so that these suprema are
at most $\varphi(B)/2$ for $L\ge L_0$, for both $X$ and $-X$, proves
\eqref{eq:lclt}.

It remains to combine these estimates. Let $n_1=\lceil8L_0/3\rceil$ and
$n\ge n_1$, so that $L\ge3n/8\ge L_0$. By \eqref{eq:deltahalfspace} and
$U=U_J+W$,
\[
 \Delta_{[\kappa,\infty)}(\sigma,\tau)
 =\PP\bigl(U_J\in[\kappa\sqrt n-W-|V|,\ \kappa\sqrt n-W+|V|)\bigr).
\]
Let $\mathcal E$ be the event that $a\sqrt k\le|V|\le b\sqrt k$,
$|V|\ge\iota/2$ and $|W|\le\sqrt n$. By \eqref{eq:uniform-prob-frombelow},
\eqref{eq:W-small} and the independence of $V$ and $W$,
$\PP(\mathcal E)\ge\rho/2$. On $\mathcal E$, the interval in the last display
has length $2|V|\ge\iota$, so it contains $\lfloor2|V|/\iota\rfloor\ge|V|/\iota$
disjoint intervals of the form $[x,x+\iota)$. Every point of the interval has
absolute value at most
\[
 |\kappa|\sqrt n+\sqrt n+b\sqrt k\le\Bigl(|\kappa|+1+\frac b2\Bigr)\sqrt n
 \le B\sqrt L,
\]
because $k\le n/4$ and $n\le8L/3$. Since $U_J$ is independent of $(V,W)$,
\eqref{eq:lclt} shows that, conditionally on $(V,W)$ and on $\mathcal E$, the
probability that $U_J$ lies in the interval is at least
$|V|\varphi(B)/(2\sqrt L)\ge a\sqrt k\,\varphi(B)/(2\sqrt L)$. Therefore
\[
 \Delta_{[\kappa,\infty)}(\sigma,\tau)
 \ge\frac{\rho}{2}\cdot\frac{a\sqrt k\,\varphi(B)}{2\sqrt L}
 \ge\frac{\rho a\varphi(B)}{4}\sqrt{\frac kn},
\]
where the last step uses $L\le n$. This proves the lemma with
$\beta_0=\rho a\varphi(B)/4$.
\end{proof}

\subsection{Transferring the lower bound to \texorpdfstring{$K$}{K}}
\label{sec:general}

In this section, we transfer \cref{lem:halfline} to the set $K$. 

\begin{lemma}\label[lemma]{lem:comparison}
Assume \eqref{eq:moments} and \eqref{eq:boundary}. There exist
$\delta_0\in(0,1/4]$, $\beta>0$ and $n_2\in\mathbb N$, depending only on $K$
and $\nu$, such that every $n\ge n_2$ and all $\sigma,\tau\in\Sigma_n$ with
$1\le d_H(\sigma,\tau)\le\delta_0n$ satisfy
\[
 \Delta_K(\sigma,\tau)\ge\beta\sqrt{d_H(\sigma,\tau)/n}.
\]
\end{lemma}

\begin{proof}
Choose $\kappa$ and $\eta$ as in \eqref{eq:boundary}, and let $F$ be the
half-line with endpoint $\kappa$ that agrees with $K$ on
$(\kappa-\eta,\kappa+\eta)$, including at $\kappa$. Namely, under the first
alternative in \eqref{eq:boundary}, $F=(-\infty,\kappa]$ if $\kappa\in K$ and
$F=(-\infty,\kappa)$ otherwise. Under the second, $F=[\kappa,\infty)$ if
$\kappa\in K$ and $F=(\kappa,\infty)$ otherwise. We first note that there are $\beta_1>0$ and $n_2\in\mathbb N$, depending only
on $\kappa$ and $\nu$, such that every $n\ge n_2$ and all $\sigma,\tau$ with
$1\le k=d_H(\sigma,\tau)\le n/4$ satisfy
\begin{equation}\label{eq:halflineF}
 \Delta_F(\sigma,\tau)\ge\beta_1\sqrt{\frac kn}.
\end{equation}
Since $\Delta_F=\Delta_{\mathbb R\setminus F}$, the half-lines
$(-\infty,\kappa)$ and $(\kappa,\infty)$ reduce to $[\kappa,\infty)$ and
$(-\infty,\kappa]$. Replacing $g$ by $-g$ turns $\Delta_{(-\infty,\kappa]}$
for $\nu$ into $\Delta_{[-\kappa,\infty)}$ for the law $\check\nu$ of $-X$,
which also satisfies \eqref{eq:moments}. Hence \eqref{eq:halflineF} follows
from \cref{lem:halfline} applied to $(\nu,\kappa)$ and $(\check\nu,-\kappa)$.

Next we compare $\Delta_K$ with $\Delta_F$. Suppose that $F$ contains exactly one of the points $(U\pm V)/\sqrt n$ and that $2|V|<\eta\sqrt n$. Then $\kappa$ lies between the two points, which are at distance
$2|V|/\sqrt n<\eta$ from each other. Hence both points lie in
$(\kappa-\eta,\kappa+\eta)$, where $K$ agrees with $F$, so $K$ also contains exactly one of them. Since $V$ is centered with variance $k$, Chebyshev's
inequality and \eqref{eq:halflineF} give
\begin{equation}\label{eq:comparison}
 \Delta_K(\sigma,\tau)
 \ge\Delta_F(\sigma,\tau)-\PP\bigl(2|V|\ge\eta\sqrt n\bigr)
 \ge\beta_1\sqrt{\frac kn}-\frac{4k}{\eta^2n}.
\end{equation}
Finally, choose $\delta_0\in(0,1/4]$ with $4\sqrt{\delta_0}/\eta^2\le\beta_1/2$.
For $k\le\delta_0n$, we have
$$4k/(\eta^2n)\le(4\sqrt{\delta_0}/\eta^2)\sqrt{k/n}\le(\beta_1/2)\sqrt{k/n}.$$
Plugging this into \eqref{eq:comparison} proves the lemma with
$\beta=\beta_1/2$.
\end{proof}

\subsection{Completing the proof}
\begin{proof}[Proof of \cref{thm:main}]
Let $\delta_0$, $\beta$ and $n_2$ be given by \cref{lem:comparison}. Since
$\sqrt\delta\log(e/\delta)\to0$ as $\delta\downarrow0$, we can choose
$\delta\in(0,\delta_0]$ such that
\begin{equation}\label{eq:deltachoice}
 \sqrt\delta\,\log(e/\delta)\le\frac{\alpha\beta}8 .
\end{equation}
Let $n\ge\max\{n_2,2/\alpha,1/\delta\}$ and $r=\lfloor\delta n\rfloor$, so
that $r\ge1$ and $m\ge\alpha n-1\ge\alpha n/2$. Recall that
$\Delta(\sigma,\tau)=\Delta_K(\sigma,\tau)$. Thus \eqref{eq:poimeanm} gives
\begin{equation}\label{eq:sumk}
 \EE R_r(G)\le3\sqrt m\,\max_{\sigma\in\Sigma_n}\sum_{k=1}^r
 \sum_{\tau:\,d_H(\sigma,\tau)=k}e^{-m\Delta_K(\sigma,\tau)/2}.
\end{equation}

Fix $\sigma\in\Sigma_n$ and $1\le k\le r$. Since $k\le\delta n\le\delta_0n$
and $m\ge\alpha n/2$, \cref{lem:comparison} gives
$m\Delta_K(\sigma,\tau)/2\ge(\alpha\beta/4)\sqrt{nk}$ for every $\tau$ with
$d_H(\sigma,\tau)=k$. There are $\binom nk\le(en/k)^k$ such $\tau$. Note that
$t\mapsto\sqrt t\log(e/t)$ is nondecreasing on $(0,e^{-1}]$, which contains
$k/n$ and $\delta$. Together with \eqref{eq:deltachoice}, this gives
\[
 \log\binom nk\le\sqrt{nk}\cdot\sqrt{\frac kn}\,\log\frac{en}k
 \le\sqrt{nk}\cdot\sqrt\delta\,\log\frac e\delta
 \le\frac{\alpha\beta}8\sqrt{nk}.
\]
Therefore
\[
 \sum_{\tau:\,d_H(\sigma,\tau)=k}e^{-m\Delta_K(\sigma,\tau)/2}
 \le\binom nk e^{-\frac{\alpha\beta}4\sqrt{nk}}
 \le e^{-\frac{\alpha\beta}8\sqrt{nk}}
 \le e^{-\frac{\alpha\beta}8\sqrt n}.
\]
Plugging this into \eqref{eq:sumk} and using $r\le n$, we get
\[
 \EE R_r(G)\le3n\sqrt m\,e^{-\frac{\alpha\beta}8\sqrt n}.
\]

Let $c=\alpha\beta/32$. Since $3n\sqrt m\,e^{-4c\sqrt n}\le e^{-2c\sqrt n}$ for
all large $n$, there is $n_0\ge\max\{n_2,2/\alpha,1/\delta\}$, depending only
on $\alpha$, $K$ and $\nu$, such that $\EE R_r(G)\le e^{-2c\sqrt n}$ for every
$n\ge n_0$. By Markov's inequality,
\[
 \PP\bigl(R_r(G)>e^{-c\sqrt n}\bigr)\le e^{c\sqrt n}\,\EE R_r(G)
 \le e^{-c\sqrt n}.
\]
The proof is complete.
\end{proof}

\section*{Statement of AI use}

The first version of this paper \cite[v1]{gong2026strong} proved strong freezing
for the ABP with Gaussian disorder. That project grew out of the authors'
ongoing work on strong freezing in related models. During its development, the
authors interacted extensively with GPT-6, posing the problem to it and
providing it with a set of their own notes from earlier work and from the
related models. Working from these notes, GPT-6 suggested replacing $Z(G)$ by a
restricted partition function that counts the solutions with nonnegative
overlap with the planted solution. Unlike $Z(G)$, this restricted partition
function is monotone in the planted direction, which allowed a monotonicity
argument to complete the proof in \cite[v1]{gong2026strong}.

The present version started with the goal of extending the result of
\cite[v1]{gong2026strong} for the ABP to general disorder. In this process, GPT-6
proposed a new way of applying the FKG inequality: replacing the number of
constraints by a Poisson variable, which induces additional independence
between the constraints and allows the FKG inequality to decouple the full
partition function from the event that two nearby configurations are both
solutions. After digesting this argument, the authors found that it extends
to general sets $K$, which led to the main result of the present version and
covers both the ABP and the SBP. In addition, the authors used other large
language models, including ChatGPT and Claude, to polish and edit the writing.
The authors verified all mathematical content themselves, wrote the complete proof in its present form, and take full responsibility for the content of this paper.

\section*{Acknowledgment}
SG was supported by the Institute for Foundations of Data Science at Yale University.

\small
\bibliographystyle{alphaabbr}
\bibliography{refs.bib}

\end{document}